\documentclass[11pt,letterpaper,reqno]{amsart}

\usepackage{amsmath,amssymb,amsthm,mathtools}
\usepackage{bbm}
\usepackage{enumitem}
\usepackage{esint}
\usepackage{tikz}

\usepackage[T1]{fontenc}
\usepackage[utf8]{inputenc}
\usepackage[hidelinks]{hyperref}

\theoremstyle{plain}
\newtheorem{theorem}{Theorem}

\newtheorem{lemma}[theorem]{Lemma}

\theoremstyle{definition}
\newtheorem{remark}[theorem]{Remark}

\numberwithin{equation}{section}

\renewcommand{\leq}{\leqslant}
\renewcommand{\geq}{\geqslant}
\renewcommand{\le}{\leqslant}
\renewcommand{\ge}{\geqslant}

\newcommand{\R}{\mathbb{R}}
\newcommand{\Z}{\mathbb{Z}}
\newcommand{\N}{\mathbb{N}}
\newcommand{\1}{\mathbbm{1}}
\newcommand{\abs}[1]{\lvert #1\rvert}

\newcommand{\supp}{\operatorname{supp}}
\newcommand{\dd}{\,\textup{d}}
\newcommand{\Fone}{\mathcal{F}_{1}}
\newcommand{\Ftwo}{\mathcal{F}_{2}}
\newcommand{\gridn}{\{0,1,\dots,n-1\}}

\begin{document}

\title[On hyperbolic corners and unit-area triangles]{On hyperbolic corners and unit-area triangles in planar sets of large measure}

\author[A. Bulj]{Aleksandar Bulj}
\email{aleksandar.bulj@math.hr}

\author[V. Kova\v{c}]{Vjekoslav Kova\v{c}}
\email{vjekovac@math.hr}

\address{Department of Mathematics, Faculty of Science, University of Zagreb, Bijeni\v{c}ka cesta 30, 10000 Zagreb, Croatia}

 %Harmonic analysis on Euclidean spaces - Singular and oscillatory integrals

\begin{abstract}
For large $R$, we consider measurable sets $A\subseteq [0,R]^2$ that avoid triples of points of the form $(x,y)$, $(x+t,y)$, $(x,y+1/t)$ with $x,y\in\mathbb{R}$ and $t>0$, i.e., the vertices of upward-oriented, axis-aligned right triangles of area $1/2$. We prove that the measures of such sets satisfy $|A|= O_c(R^2/(\log R)^c)$ for any constant $c<1/4$. An ingredient in the proof is a hyperbolic variant of the two-dimensional trilinear smoothing inequality by Christ, Durcik, and Roos. The aforementioned upper bound is complemented with an example of a set of measure $\Omega(R\log R)$ avoiding the same point configuration. 

Next, we study measurable sets $A\subseteq [0,R]^2$ that avoid triples of points spanning a triangle of a given fixed area and establish a sharpening of the aforementioned upper bound to any $c<1/2$. This makes partial progress on a question by Erd\H{o}s, who conjectured an upper bound $O(1)$, and improves over a quantitatively weak $o(R^2)$ result by Graham. The latter proof additionally uses induction on scales to interchangeably control the density and the Riesz energy of the set $A$.
\end{abstract}

\maketitle

\tableofcontents

%%%%%

\section{Introduction}

We refer to Subsection~\ref{subsec:notation} for the quite standard asymptotic notation that we use in this paper.

%%%%%

\subsection{Sets with no hyperbolic corners}
The first main result of the present paper deals with bounds on measures of subsets $A$ of a large square $[0,R]^2$ that avoid a certain three-point configuration, which we call a \emph{hyperbolic corner}.

\begin{theorem}\label{thm:main}
Let $M(R)$ be the supremum of $|A|$ over measurable sets $A \subseteq [0,R]^2$ containing no triples of the form
\begin{equation}\label{eq:mainpattern}
(x,y),\ (x+t,y),\ (x, y+t^{-1}) 
\end{equation}
with $x,y\in\mathbb{R}$ and $t>0$. Then for $R\geq 10$ we have
\begin{equation}\label{eq:mainestimate}
R\log R \lesssim M(R) \lesssim R^2 \Bigl(\frac{\log\log R}{\log R}\Bigr)^{1/4}. 
\end{equation}
\end{theorem}

In particular,
\[ M(R) \lesssim_{\epsilon} \frac{R^2}{(\log R)^{1/4-\epsilon}} \]
for every $0<\epsilon<1/4$.
We do not know of any published nontrivial (i.e., better than $O(R^2)$) upper estimate on $M(R)$. 
The proof of the upper estimate in \eqref{eq:mainestimate} will be given in Section~\ref{sec:corners_upper}. The exponent $1/4$ in the upper bound, like the Estimate \eqref{eq:mainestimate} itself, is unlikely to be optimal, but it reflects the limitations of currently available techniques.

Theorem~\ref{thm:main} naturally fits into an active line of research investigating \emph{nonlinear corners}, which are configurations of the form
\begin{equation}\label{eq:nonlincorner} 
(x,y),\ (x+\gamma_1(t),y),\ (x,y+\gamma_2(t))
\end{equation}
contained either in continuous structures, like $[0,1]^2$ and $\R^2$, or in discrete structures like $\gridn^2$ and $\mathbb{F}_q^2$.
These are ``curved'' counterparts of the usual (linear) corners, studied extensively since \cite{AS74,Shk06}, with the recent state of the art being \cite{JLLOS25}. 
Most of the literature is concerned with the cases when $\gamma_1$ and $\gamma_2$ are polynomials (often of distinct degrees). Bergelson and Leibman \cite{BL96} (also see \cite{BHMP00}) showed a density theorem on the integer lattice with non-effective bounds. This has been successively quantitatively improved all the way to polylogarithmic bounds like ours; see \cite{CG24,KKL24,GMZ26}. Finite-field results were being developed simultaneously \cite{DLS20,HLY21,Kuc24patterns,Kuc24distinct,KMPW24,Lim25}.

However, the present configuration \eqref{eq:mainpattern} is not a polynomial corner, which rules out purely polynomial tools, such as the PET induction and the degree lowering of Gowers norms. The paper closest in technique to ours is the groundbreaking work of Christ, Durcik, and Roos \cite{CDR21}, who showed a continuous-parameter density theorem for
\[ (x,y),\ (x+t,y),\ (x,y+t^a), \]
where $a>1$ is fixed. (For simplicity of notation, only the case $a=2$ was worked out explicitly in \cite{CDR21}.)
The proof in \cite{CDR21} does not fully extend to the case $a=-1$ (or, in fact, any $a<0$), even though we are primarily influenced by their work, which, in turn, reflects ideas from the paper by Bourgain \cite{B88nonlinear} on the nonlinear Roth theorem in $\R$.
It is also fair to remark that Lim \cite{Lim25} studied \eqref{eq:nonlincorner} for certain rational functions $\gamma_1,\gamma_2$, but in the finite-field setting $\mathbb{F}_q^2$, where rational functions are close in spirit to polynomials.

The problem of estimating $M(R)$ also belongs to the broader theme of forcing finite configurations inside large subsets of Euclidean spaces; see \cite{Kovac:survey} for a brief and accessible introduction to the topic. 
The two-point subpattern $(x+t,y)$, $(x,y+1/t)$ of \eqref{eq:mainpattern} has already attracted some attention in the literature. Indeed, after a rotation and a translation, it becomes the configuration
\begin{equation}\label{eq:BMKpattern}
(x,y), \ (x,y)+(t,t^{-1})
\end{equation}
studied by Bardestani and Mallahi-Karai \cite{BMK17} (also see \cite[Thm.~10]{Kov26}), who proved an estimate of the same polylogarithmic type as \eqref{eq:mainestimate}, namely $O(R^2/\log R)$, but under the stronger hypothesis that a set $A\subseteq[0,R]^2$ avoids every pair of the form \eqref{eq:BMKpattern}.
Two-point patterns of the form
\[ (x,y), \ (x,y)+(t,\gamma(t)) \]
for appropriate curves $\gamma$ have also been studied in \cite{KOS23,DKS23,CM24,BP25}.
Configurations with more points typically require additional ideas, and known approaches often work in higher ambient dimensions \cite{B86:roth,LM16:prod,LM19:hypergraphs,LM20}. A result in that direction is \cite[Thm.~1]{Kov26}. Its special case concerning axis-aligned right triangles is much weaker than Theorem~\ref{thm:main}: it asserts that a set $A\subseteq[0,R]^3$ avoiding every triple of points
\[ (x,y,z), \ (x+t,y,z), \ (x,y+t^{-1}u,z+t^{-1}v)  \]
for $u^2+v^2=1$ satisfies the measure bound $|A|=O(R^3/(\log R)^{36})$. The important point is not the different power of the logarithm, but rather the fact that an additional degree of freedom $(u,v)\in\textup{S}^1$ allows one to use the Fourier decay of the surface measure, which is clearly not available in the plane.

The basic outline of the proof of the upper bound in \eqref{eq:mainestimate} only superficially resembles those in recent papers in geometric Ramsey theory \cite{CMP15:roth,DKR18,K20:anisotrop,DK22,Kov26}. The beginning is quite standard: we define an appropriate pattern-counting form \eqref{eq:justN0} associated with a parameter $\lambda>0$. In our case $\lambda$ determines an approximate scale of eccentricity of the configuration \eqref{eq:mainpattern} that we aim to detect. A certain novelty (employed also in \cite{Kov26}) is that this parameter does not represent any usual notion of configuration ``size.'' In particular, it makes no sense to let $\lambda\rightarrow0^+$ or $\lambda\rightarrow+\infty$. In classical scaling arguments, letting $\lambda$ approach $0$ simply means that one also detects arbitrarily small instances of the desired configuration, which are in fact trivially found by the Lebesgue density theorem. Despite the absence of this interpretation, we can still meaningfully decompose the counting form into the structured, error, and uniform parts \eqref{eq:Ndecomposition}, just as in, e.g., \cite{CMP15:roth,K20:anisotrop,DK22}.
Another notable difference here is that all values of $\lambda$ contribute equally to the overall pattern count. Thus, to control the error part, we can simply integrate in $\lambda$ over $[1/R,R]$ and avoid using Bourgain's trick of pigeonholing in the scales \cite{B86:roth,B88nonlinear}. This additional symmetry is responsible for the better upper bound in \eqref{eq:mainestimate} compared with \cite{CDR21}, with a power of $\log R$, rather than a power of $\log\log R$, in the denominator.
The final ingredient in the proof of Theorem~\ref{thm:main} is to handle the uniform part by reducing its bound to a hyperbolic variant of the trilinear smoothing estimate by Christ, Durcik, and Roos \cite{CDR21}; see Lemma~\ref{lm:smoothing}. 

A trivial lower bound for $M(R)$ is $M(R) \gtrsim R$, and it is realized by a single skew strip of width $\sim 1$ intersecting the square $[0,R]^2$. The proof of the lower estimate in \eqref{eq:mainestimate} uses an example made up of several strips of varying width. It is depicted in Figure~\ref{fig:exampleset}, and the details are presented in Section~\ref{sec:lower_bound}.

%%%%%

\subsection{Sets with no unit-area triangles}
Observe that the configuration \eqref{eq:mainpattern} consists of the vertices of an upward-oriented, axis-aligned right triangle of area $1/2$. We turn to the study of sets that avoid the vertices of any (possibly rotated and scalene) triangle of a fixed area. Part of the motivation comes from the following problem by Erd\H{o}s on unit-area triangles, as quoted from \cite[Prob.~352]{EP}.
 
\begin{quote}
\emph{Is it true that there is an absolute constant $C$ so that, if $A\subseteq\R^2$ has area greater than $C$, then $A$ contains the vertices of a triangle of area $1$?}
\end{quote}

Erd\H{o}s posed this problem in \cite[p.~122--123]{Erd78}, \cite[p.~30]{Erd:Scottish}, and \cite[p.~323]{Erd83:open}, and the problem was also popularized by Croft, Falconer, and Guy \cite[Prob.~G13]{CFG91}, Mauldin \cite[Sec.~1]{Mau02}, \cite[Sec.~5]{Mau13}, and, most recently, the \emph{Erd\H{o}s problems} community \cite[Prob.~352]{EP}.
Our second main result makes partial progress by proving an upper bound on measures of sets $A$ contained in a large square $[0,R]^2$ that do not contain vertices of a unit-area triangle.

\begin{theorem}\label{thm:triangles}
Let $M_{\Delta}(R)$ be the supremum of $|A|$ over measurable sets $A\subseteq [0,R]^2$ containing no triples of points that span a triangle of area $1$. Then for $R\geq 10$ we have
\begin{equation}\label{eq:mainestimate2}
M_{\Delta}(R) \lesssim R^2 \Bigl(\frac{\log\log R}{\log R}\Bigr)^{1/2}. 
\end{equation}
\end{theorem}

Clearly, $M_{\Delta}(R)\lesssim M(R)$, but a more involved proof will give us the upgraded Estimate \eqref{eq:mainestimate2}. In particular, we get
\[ M_{\Delta}(R) \lesssim_{\epsilon} \frac{R^2}{(\log R)^{1/2-\epsilon}} \]
for every $\epsilon>0$.
The proof of Theorem~\ref{thm:triangles} will use a couple of ingredients that might be entirely novel in the context of Euclidean Ramsey theory. The first one is the replacement of the set measure $|A|$ by the smaller quantity of the Riesz energy $\mathcal{E}(A)\lesssim|A|^{3/2}$, see \eqref{eq:def_Riesz_en}, in order to obtain quantitative gain. The second one is the induction on scales $R$, which interchangeably controls $\mathcal{E}(A)$ with $|A|$, and vice versa.

We do not know of any non-trivial lower bounds and the original Erd\H{o}s problem is still open. In fact, Erd\H{o}s even suspected that $M_{\Delta}(R)=4\pi/\sqrt{27}$ for every large $R$, witnessed by the open disk circumscribed around the equilateral triangle of area $1$.
To the best of our knowledge, there are also no published upper bounds on $M_{\Delta}(R)$ of the shape $o(R^2)$ as $R\to\infty$. However, the corresponding discrete density result was commented rather cryptically by Graham \cite[p.\,96]{Gra80}, in relation with his coloring result for triangles.

\begin{quote}
\emph{There is a \textit{density} version of Theorem 1 which can be proved, based on using Szemer\'{e}di's theorem in place of van der Waerden's theorem. This density version asserts that for any $\epsilon > 0$, there is an integer $T(\epsilon)$ so that if $n \geqslant n(\epsilon)$ and $R \subseteq \{(i,j) : 1 \leqslant i, j \leqslant n\}$ with $|R| > \epsilon n^2$, then $R$ contains the vertices of a triangle of area $T(\epsilon)$. We will not discuss those directions in this note, however.}
\end{quote}

In Appendix~\ref{sec:Grahams} we decipher what Graham actually meant and use this to derive 
\begin{equation}\label{eq:Graham_weak}
M_{\Delta}(R)=o(R^2)\quad\text{as } R\to\infty.
\end{equation} 
However, the latter proof does not come anywhere near concrete effective estimates on $M_{\Delta}(R)$.

%%%%%

\subsection{Notation}
\label{subsec:notation}
Let $\mathcal{A}\lesssim\mathcal{B}$, $\mathcal{B}\gtrsim\mathcal{A}$, $\mathcal{A}=O(\mathcal{B})$, and $\mathcal{B}=\Omega(\mathcal{A})$ all mean that there exists a positive constant $C$ such that $|\mathcal{A}|\leq C|\mathcal{B}|$. If the constant is allowed to depend on a set of parameters $P$, we emphasize this as $\mathcal{A}\lesssim_P\mathcal{B}$, $\mathcal{B}\gtrsim_P\mathcal{A}$, $\mathcal{A}=O_P(\mathcal{B})$, or $\mathcal{B}=\Omega_P(\mathcal{A})$.
The meaning of $\mathcal{A}\sim\mathcal{B}$ (resp.\@ $\mathcal{A}\sim_P\mathcal{B}$) is that both $\mathcal{A}\lesssim\mathcal{B}$ and $\mathcal{A}\gtrsim\mathcal{B}$ (resp.\@ $\mathcal{A}\lesssim_P\mathcal{B}$ and $\mathcal{A}\gtrsim_P\mathcal{B}$) hold.

We write $|A|$ for the Lebesgue measure of a measurable subset $A$ of the Euclidean plane. We also write $\textup{D}(z,R)$ for the Euclidean planar disk of radius $R$ centered at $z\in\R^2$. 

The \emph{Fourier transform} of an integrable function $f\colon\R^2\to\mathbb{C}$ will be normalized as
\[ \widehat{f}(\xi,\eta) := \int_{\R^2} f(x,y) e^{-2\pi i(x\xi+y\eta)} \dd x \dd y, \quad (\xi,\eta)\in\R^2. \]
Throughout the paper, $\Fone$ will denote the Fourier transform in the first variable and $\Ftwo$ the Fourier transform in the second variable only:
\[ (\mathcal{F}_1 f)(\xi,y) := \int_{\R} f(x,y) e^{-2\pi ix\xi} \dd x, \quad
(\mathcal{F}_2 f)(x,\eta) := \int_{\R} f(x,y) e^{-2\pi iy\eta} \dd y . \]
\emph{Convolution} of integrable functions $f,g\colon\R^2\to\mathbb{C}$ is defined as
\[ (f\ast g)(x,y) := \int_{\R^2} f(x-u,y-v) g(u,v) \dd u \dd v \]
for almost every $(x,y)\in\R^2$.
Also, $\ast_1$ (resp.\@ $\ast_2$) will denote the convolution of a two-variable function $f$ in the first (resp.\@ second) variable with a given single-variable function $\varphi$:
\[ (f\ast_1 \varphi)(x,y) := \int_{\R} f(x-s,y) \varphi(s)\dd s,\quad
(f\ast_2 \varphi)(x,y) := \int_{\R} f(x,y-s) \varphi(s)\dd s. \]

%%%%%

\section{Proof of the upper bound in Theorem~\ref{thm:main}}
\label{sec:corners_upper}

Fix a nonnegative $C^\infty$ function $\zeta$ that is not identically zero and has compact support contained in the interval $[1/2,2]$. Also let $\phi$ be an even nonnegative $C^\infty$ function with $\int_{\R} \phi =1$, supported in $[-20,20]$ and such that $\phi$ is bounded from below by a positive constant on $[-10,10]$.
For $t>0$ write
\[ \phi_t(x):=\frac{1}{t}\phi\Bigl(\frac{x}{t}\Bigr). \]

If $f_0,f_1,f_2\colon \R^2\to\R$ are bounded, compactly supported, and measurable, define the \emph{exact counting form} associated with a certain scaling parameter $\lambda\in(0,\infty)$ by
\begin{equation}\label{eq:justN0}
\mathcal{N}^0_\lambda(f_0,f_1,f_2) := \int_0^{\infty}\int_{\R^2} f_0(x,y) f_1(x+\lambda u,y) f_2\Bigl(x,y+\frac{1}{\lambda u}\Bigr) \zeta(u)\dd x\dd y\dd u. 
\end{equation}
In the case $\lambda=1$ we simply write $\mathcal{N}$ in place of $\mathcal{N}^0_1$:
\begin{equation}\label{eq:justN}
\mathcal{N}(f_0,f_1,f_2) := \int_0^{\infty}\int_{\R^2} f_0(x,y) f_1(x+u,y) f_2\Bigl(x,y+\frac{1}{u}\Bigr) \zeta(u)\dd x\dd y\dd u.
\end{equation}
For $0<\varepsilon\leq 1$, we also introduce the \emph{smoothed counting form}
\[ \mathcal{N}^\varepsilon_\lambda(f_0,f_1,f_2) := \mathcal{N}^0_\lambda\Bigl(f_0,f_1\ast_1\phi_{\lambda\varepsilon},f_2\ast_2\phi_{\lambda^{-1}\varepsilon}\Bigr). \]
The dominated convergence theorem easily gives
\[ \lim_{\varepsilon\to0^+} \mathcal{N}^\varepsilon_\lambda(f_0,f_1,f_2) = \mathcal{N}^0_\lambda(f_0,f_1,f_2). \]
These forms correspond to a localized piece of the hyperbola $uv=1$ determined by the parameter $\lambda$, namely the piece in a scaled neighborhood of the point $(\lambda,1/\lambda)$.
The other parameter, $\varepsilon$, is viewed as the smoothing parameter, which ``blurs out'' the picture and thus typically overcounts the occurrences of the sought point configuration.

When $f_0=f_1=f_2=\1_A$, we simply write $\mathcal{N}^0_\lambda(A)$ and $\mathcal{N}^\varepsilon_\lambda(A)$, respectively.
The form $\mathcal{N}^0_\lambda(A)$ counts triples
\[ (x,y),\ (x+\lambda u,y),\ \Bigl(x,y+\frac{1}{\lambda u}\Bigr) \]
with $u\in\supp \zeta\subseteq [1/2,2]$. Hence, if $A$ contains no triples of the form in Theorem~\ref{thm:main}, then
\begin{equation}\label{eq:exact_zero}
\mathcal{N}^0_\lambda(A)=0
\quad\text{for every }\lambda>0.
\end{equation}
We decompose
\begin{equation}\label{eq:Ndecomposition}
\mathcal{N}^0_\lambda(A) = \mathcal{N}^1_\lambda(A) + \Bigl( \mathcal{N}^\varepsilon_\lambda(A) - \mathcal{N}^1_\lambda(A) \Bigr) + \Bigl( \mathcal{N}^0_\lambda(A) - \mathcal{N}^\varepsilon_\lambda(A) \Bigr) 
\end{equation}
for a number $\varepsilon\in(0,1)$ that will be chosen later.
In analogy with the literature in graph theory and arithmetic combinatorics, the three terms are respectively called:
\begin{itemize}
  \item $\mathcal{N}^1_\lambda(A)$ = the structured part,
  \item $\mathcal{N}^\varepsilon_\lambda(A) - \mathcal{N}^1_\lambda(A)$ = the error part, and
  \item $\mathcal{N}^0_\lambda(A) - \mathcal{N}^\varepsilon_\lambda(A)$ = the uniform part;
\end{itemize}
the same names were also used for somewhat analogous quantities in \cite{DK22}, \cite{K20:anisotrop}, and in several later papers.
The basic strategy is proving a lower bound for the structured part (i.e., on the very ``coarse'' scale $\varepsilon=1$) that is independent of $\lambda$ taken from the most reasonable scale range $\lambda\in[1/R,R]$. The uniform part should satisfy an upper bound that is again uniform in $\lambda$ and makes it negligible in comparison with the structured part as soon as $\varepsilon$ is chosen to be sufficiently small. Then we are left with the error part, which will be controlled, but rather in the $\textup{L}^2$ norm sense with respect to the measure $\dd\lambda/\lambda$. This will contradict \eqref{eq:exact_zero}, not for a single concrete value of $\lambda$, but ``on the average,'' i.e., once we integrate \eqref{eq:Ndecomposition} in $\lambda$.
The same basic strategy will also be employed in the proof of Theorem~\ref{thm:triangles}.
Once again, we remark that this strategy proves to be efficient, despite the fact that one cannot ``interpolate'' the proof with some trivial limit $\lambda\to0^+$, as is usually the case with configurations that can shrink to a point so that the Lebesgue density theorem applies there.

%%%%%

\subsection{The structured part}
We begin with the study of the first term appearing in decomposition \eqref{eq:Ndecomposition}.

\begin{lemma}\label{lm:structured}
For every measurable $A\subseteq[0,R]^2$ and every $\lambda$ with
\[ R\geq \max\{\lambda,\lambda^{-1}\}, \]
one has
\[ \mathcal{N}^1_\lambda(A) \gtrsim \frac{\abs{A}^3}{R^4}
= R^2\Bigl(\frac{\abs{A}}{R^2}\Bigr)^3. \]
\end{lemma}

\begin{proof}
First, we rewrite the smoothed form with $\varepsilon=1$ as
\[ \mathcal{N}^1_\lambda(f_0,f_1,f_2) = \int_{\R^4} f_0(x,y) f_1(x',y) f_2(x,y') K^+_\lambda (x-x',y-y') \dd x \dd x' \dd y \dd y', \]
where
\begin{equation}\label{eq:Kplus}
K^+_\lambda(x,y) := \int_0^{\infty} \phi\Bigl(\frac{x}{\lambda}+u\Bigr) \phi\Bigl(\lambda y+\frac{1}{u}\Bigr) \zeta(u) \dd u. 
\end{equation}
Since $\phi$ is bounded from below on $[-10,10]$ and $\zeta$ is supported in $[1/2,2]$, we have the pointwise kernel bound
\begin{equation}\label{eq:Kplusbd}
K^+_\lambda \gtrsim \1_{[-\lambda,\lambda]\times[-\lambda^{-1},\lambda^{-1}]},
\end{equation}
which we can plug into the last representation of $\mathcal{N}^1_\lambda(A)=\mathcal{N}^1_\lambda(\1_A,\1_A,\1_A)$.
That way we obtain 
\begin{equation}\label{eq:smoothedN1}
\mathcal{N}^1_\lambda(A) \gtrsim \sum_{Q=I\times J} \int_{I\times I\times J\times J} \1_A(x,y)\1_A(x',y)\1_A(x,y')\dd x\dd x'\dd y\dd y', 
\end{equation}
where the sum runs over the partition of $[0,R]^2$ into $N\sim R^2$ many rectangles $Q=I\times J$ with $\abs{I}\leq \lambda$ and $\abs{J}\leq 1/\lambda$.

Fix one such rectangle $Q=I\times J$. Abbreviate
\[ F(x):=\int_J \1_{A}(x,y') \dd y', \quad G(y):=\int_I \1_{A}(x',y) \dd x'. \]
By H\"older's inequality,
\begin{align*}
& |A\cap Q|^3 = \Bigl(\int_{I\times J} \1_{A}(x,y) \dd x \dd y \Bigr)^3 \\
& \leq \Bigl(\int_{I\times J} \1_{A}(x,y)F(x)G(y) \dd x \dd y \Bigr)
\Bigl(\underbrace{\int_{I\times J} \frac{\1_{A}(x,y)}{F(x)} \dd x \dd y}_{\leq|I|}\Bigr)
\Bigl(\underbrace{\int_{I\times J} \frac{\1_{A}(x,y)}{G(y)} \dd x \dd y}_{\leq|J|}\Bigr).
\end{align*}
Here we interpret $\1_{A}(x,y)/F(x)$ and $\1_{A}(x,y)/G(y)$ as $0$ wherever the denominator vanishes.
Hence, the corresponding summand in \eqref{eq:smoothedN1} satisfies
\[ \int_{I\times J} \1_{A}(x,y)F(x)G(y) \dd x \dd y
\geq \frac{|A\cap Q|^3}{|I||J|} \geq \abs{A\cap Q}^3, \]
since $|I||J|\leq \lambda\lambda^{-1}=1$.
By convexity of $t\mapsto t^3$ and Jensen's inequality,
\[ \sum_Q |A\cap Q|^3 \geq N^{-2}\Bigl(\sum_Q |A\cap Q|\Bigr)^3 \gtrsim \frac{|A|^3}{R^4}. \]
Combining the last two displayed estimates proves the claim.
\end{proof}

%%%%%

\subsection{The error part}
The second needed ingredient is an $\textup{L}^2$ control of the error part from \eqref{eq:Ndecomposition}.

\begin{lemma}\label{lm:L2error}
For every bounded measurable $A\subseteq\R^2$ and every $0<\varepsilon\leq 1$,
\begin{equation}\label{eq:L2error}
\Bigl( \int_0^{\infty} \abs{\mathcal{N}^\varepsilon_\lambda(A)-\mathcal{N}^1_\lambda(A)}^2 \,\frac{\dd\lambda}{\lambda} \Bigr)^{1/2}
\lesssim |A| \Bigl(\log\frac{1}{\varepsilon}\Bigr)^{1/2}.
\end{equation}
\end{lemma}

\begin{proof}
Using
\[ \phi_{\lambda\varepsilon}(x) \phi_{\lambda^{-1}\varepsilon}(y) - \phi_\lambda(x) \phi_{\lambda^{-1}}(y)
= (\phi_{\lambda\varepsilon}-\phi_\lambda)(x) \,\phi_{\lambda^{-1}\varepsilon}(y)
+ \phi_\lambda(x) \,(\phi_{\lambda^{-1}\varepsilon} -\phi_{\lambda^{-1}})(y) \]
the form definitions yield
\[ \mathcal{N}^\varepsilon_\lambda(A)-\mathcal{N}^1_\lambda(A)
= E_{1}(\lambda,\varepsilon) + E_{2}(\lambda,\varepsilon), \]
where
\begin{align*}
E_{1}(\lambda,\varepsilon)
&:= \mathcal{N}^0_\lambda\Bigl(\1_A, \1_A\ast_1(\phi_{\lambda\varepsilon}-\phi_{\lambda}), \1_A\ast_2\phi_{\lambda^{-1}\varepsilon}\Bigr), \\
E_{2}(\lambda,\varepsilon)
&:= \mathcal{N}^0_\lambda\Bigl(\1_A, \1_A\ast_1\phi_{\lambda}, \1_A\ast_2(\phi_{\lambda^{-1}\varepsilon} -\phi_{\lambda^{-1}})\Bigr).
\end{align*}
For the first term, we use 
$0\leq \1_A\ast_2\phi_{\lambda^{-1}\varepsilon}\leq1$.
For each fixed $u$ and $\lambda$, the Cauchy--Schwarz inequality in $(x,y)$ gives
\[ \int_{\R^2} \1_{A}(x,y)\,
\big|\Bigl(\1_{A}\ast_1(\phi_{\lambda\varepsilon}-\phi_{\lambda})\Bigr)(x+\lambda u,y)\big| \dd x\dd y
\leq \|\1_{A}\|_{\textup{L}^2}\,\big\|\1_{A}\ast_1(\phi_{\lambda\varepsilon}-\phi_{\lambda})\big\|_{\textup{L}^2}, \]
so that, integrating this estimate against $\zeta(u)\dd u$, we obtain
\[ |E_1(\lambda,\varepsilon)| \lesssim |A|^{1/2} \big\|\1_A\ast_1(\phi_{\lambda\varepsilon}-\phi_{\lambda})\big\|_{\textup{L}^2}. \]
Consequently, by Plancherel's theorem,
\begin{align*}
\int_0^\infty |E_1(\lambda,\varepsilon)|^2\frac{\dd\lambda}{\lambda}
& \lesssim |A| \int_0^\infty \big\|\1_A\ast_1(\phi_{\lambda\varepsilon}-\phi_{\lambda})\big\|_{\textup{L}^2}^2 \frac{\dd\lambda}{\lambda} \\
& \lesssim |A|\int_{\R^2} |\widehat{\1_A}(\xi,\eta)|^2
\biggl(\int_0^\infty \big|\widehat{\phi}(\varepsilon\lambda\xi)-\widehat{\phi}(\lambda\xi)\big|^2\frac{\dd\lambda}{\lambda}\biggr)\dd\xi\dd\eta \\
& \lesssim |A|\log\frac1{\varepsilon}\int_{\R^2}|\widehat{\1_A}(\xi,\eta)|^2\dd\xi\dd\eta
= |A|^2\log\frac1{\varepsilon}.
\end{align*}
Namely, the inner integral is $\lesssim\log(1/\varepsilon)$ uniformly in $\xi\neq0$, which is seen by changing variables $s=\lambda|\xi|$ and observing
\begin{equation}\label{eq:thephitrick}
\int_0^\infty \big|\widehat{\phi}(\varepsilon s)-\widehat{\phi}(s)\big|^2 \frac{\dd s}{s} \lesssim \log\frac1{\varepsilon}. 
\end{equation}
Estimate \eqref{eq:thephitrick} is, in turn, a consequence of
\[ \big|\widehat{\phi}(\varepsilon s)-\widehat{\phi}(s)\big| \lesssim 
\begin{cases}
s & \text{for } 0<s\leq 1, \\
1 & \text{for } 1<s\leq \varepsilon^{-1}, \\
(\varepsilon s)^{-1} & \text{for } s>\varepsilon^{-1},
\end{cases} \]
which follows from $\widehat{\phi}(0)=1$, smoothness, and the rapid decay of $\widehat{\phi}$.
The same argument in the second variable gives
\[ \int_0^\infty |E_2(\lambda,\varepsilon)|^2 \frac{\dd\lambda}{\lambda}
\lesssim |A|^2\log\frac1{\varepsilon}, \]
which finally proves \eqref{eq:L2error}.
\end{proof}

%%%%%

\subsection{The uniform part}

For $s_1,s_2\in\R$ we use the anisotropic Sobolev norm
\[ \|f\|_{\textup{H}^{s_1,s_2}} := \Bigl( \int_{\R^2} \abs{\widehat f(\xi,\eta)}^2 \Bigl(1+\xi^2\Bigr)^{s_1} \Bigl(1+\eta^2\Bigr)^{s_2} \dd \xi\dd \eta \Bigr)^{1/2}. \]
A genuinely non-elementary ingredient in the paper is the following local estimate, which is a hyperbolic analogue of the trilinear smoothing inequality of Christ, Durcik, and Roos \cite{CDR21}.

\begin{lemma}[Hyperbolic trilinear smoothing]\label{lm:smoothing}
There exists $\sigma>0$ such that the trilinear form \eqref{eq:justN} satisfies
\[ |\mathcal{N}(f_0,f_1,f_2)|
\lesssim \|f_0\|_{\textup{L}^{\infty}} \|f_1\|_{\textup{H}^{-\sigma,0}} \|f_2\|_{\textup{H}^{0,-\sigma}} \]
for bounded measurable functions $f_0,f_1,f_2\colon \R^2\to\R$ supported in $[-10,10]^2$.
\end{lemma}

Lemma~\ref{lm:smoothing} will be proved in Section~\ref{sec:smoothing_proof} of the appendix. Here we return to decomposition \eqref{eq:Ndecomposition} and manage to control its last term.

\begin{lemma}\label{lm:uniform}
There exists $\sigma>0$ such that, for every bounded measurable set $A\subseteq\R^2$, every $\lambda>0$, and every $0<\varepsilon\leq 1$,
\[ \abs{\mathcal{N}^0_\lambda(A)-\mathcal{N}^\varepsilon_\lambda(A)} \lesssim |A|\varepsilon^{\sigma}. \]
\end{lemma}

\begin{proof}
Since $\lambda>0$ is fixed, we can define the anisotropically rescaled indicator function
\[ f(x,y):=\1_A\Bigl(\lambda x, \frac{y}{\lambda}\Bigr), \]
which satisfies
\begin{equation*}
\|f\|_{\textup{L}^\infty}\leq 1, \quad \|f\|_{\textup{L}^2} = |A|^{1/2}.
\end{equation*}
We will make sure that the estimates below depend on no other properties of $f$, and in particular do not depend on $\lambda$.
This is a place in our proof where the scaling $(t,t^{-1})$ comes as very convenient.
Recall the definitions of the counting forms. Substitutions $x=\lambda x'$, $y=y'/\lambda$ give
\[ \mathcal{N}^0_\lambda(A) = \mathcal{N}(f, f, f) \]
and
\[ \mathcal{N}^\varepsilon_\lambda(A) = \mathcal{N}\bigl(f,\,f\ast_1\phi_\varepsilon,\,f\ast_2\phi_\varepsilon\bigr) \]
for $0<\varepsilon\leq1$.
Our goal is to prove a slightly more general estimate than needed:
\begin{equation}\label{eq:more_general}
\big| \mathcal{N}(f, g, h) - \mathcal{N}\bigl(f,\,g\ast_1\phi_\varepsilon,\,h\ast_2\phi_\varepsilon\bigr) \big| \lesssim \varepsilon^\sigma \|f\|_{\textup{L}^\infty} \|g\|_{\textup{L}^2} \|h\|_{\textup{L}^2},
\end{equation}
which will prove the lemma, once one takes $g$ and $h$ to be the same as $f$ defined above.
The purpose of this generalization is to reuse \eqref{eq:more_general} later, in the proof of Lemma~\ref{lm:uniform2} relevant for Theorem~\ref{thm:triangles}.

First, we decompose the difference on the left-hand side of \eqref{eq:more_general} as $T_1(\varepsilon) + T_2(\varepsilon)$, where
\begin{align*}
T_1(\varepsilon) &:= \mathcal{N}\bigl(f, \,g - g\ast_1\phi_\varepsilon, \,h\bigr),\\
T_2(\varepsilon) &:= \mathcal{N}\bigl(f, \,g\ast_1\phi_\varepsilon, \,h - h\ast_2\phi_\varepsilon\bigr).
\end{align*}
We now localize the two terms $T_1$ and $T_2$. The point of the localization is only to put the local smoothing estimate on bounded ``windows,'' but no estimate on the number of windows will be used. All constants in this localization argument are independent of $R$, $\lambda$, and $\varepsilon$.
Choose a partition of unity on $\R^2$,
\[ 1 = \sum_{m\in\Z^2} \rho(z-m), \quad z\in\R^2, \]
for some $C^\infty$ function $\rho$ supported in the square $[-2,2]^2$.
For $m\in\Z^2$ put
\[ f_m(z) := f(z)\rho(z-m),\quad
g_m(z) := g(z)\rho(z-m),\quad
h_m(z) := h(z)\rho(z-m), \]
so that only finitely many pieces $f_m,g_m,h_m$ are nonzero and bounded overlap gives
\begin{equation}\label{eq:local-L2-overlap}
\sum_{m\in\Z^2} \|g_m\|_{\textup{L}^2}^2\lesssim \|g\|_{\textup{L}^2}^2,\quad
\sum_{m\in\Z^2} \|h_m\|_{\textup{L}^2}^2\lesssim \|h\|_{\textup{L}^2}^2.
\end{equation}

We expand the first term as
\begin{equation}\label{eq:T1-localized-sum}
T_1(\varepsilon) = \sum_{k\in\Z^2}\sum_{l\in\Z^2}\sum_{m\in\Z^2} \mathcal{N}\bigl(f_m, g_{m+k}-g_{m+k}\ast_1\phi_\varepsilon, h_{m+l}\bigr),
\end{equation}
noting that the sums in $k$ and $l$ are both effectively finite, since only finitely many summands are nonzero.
In fact, the integrand
\[ f_m(x,y) \,(g_{m+k}-g_{m+k}\ast_1\phi_\varepsilon)(x+u,y) \,h_{m+l}\Bigl(x,y+\frac{1}{u}\Bigr) \,\zeta(u) \]
is identically zero unless absolute values of both coordinates of $k$ and $l$ are at most $5$, which amounts to having to sum only over $O(1)$ many pairs of indices $(k,l)$. For that reason, it is sufficient to bound the innermost sum in \eqref{eq:T1-localized-sum}, the one in $m$, uniformly in the fixed indices $k$ and $l$.
We next apply Lemma~\ref{lm:smoothing} to each such summand, with the smoothing exponent $\sigma\in(0,1]$ as in the lemma.
Its estimate applies to the functions $f_m$, $g_{m+k}-g_{m+k}\ast_1\phi_\varepsilon$, $h_{m+l}$, all translated by $-m$ (so that the translates become supported in $[-10,10]^2$), giving
\[ \big|\mathcal{N}\bigl(f_m, g_{m+k}-g_{m+k}\ast_1\phi_\varepsilon, h_{m+l}\bigr)\big|
\lesssim \|f_m\|_{\textup{L}^\infty} \|g_{m+k}-g_{m+k}\ast_1\phi_\varepsilon\|_{\textup{H}^{-\sigma,0}} \|h_{m+l}\|_{\textup{H}^{0,-\sigma}}. \]
Here we used the fact that the Sobolev norms remain unchanged by the translation.
Using the Cauchy--Schwarz inequality in $m$, we obtain
\begin{align}
& \Big| \sum_{m\in\Z^2} \mathcal{N}\bigl(f_m, g_{m+k}-g_{m+k}\ast_1\phi_\varepsilon, h_{m+l}\bigr) \Big| \nonumber \\
& \lesssim \|f\|_{\textup{L}^\infty} \Bigl(\sum_{m\in\Z^2} \|g_m-g_m\ast_1\phi_\varepsilon\|_{\textup{H}^{-\sigma,0}}^2\Bigr)^{1/2}
\Bigl(\sum_{m\in\Z^2} \|h_m\|_{\textup{H}^{0,-\sigma}}^2\Bigr)^{1/2}. \label{eq:rhsoftwob}
\end{align}
Since $\widehat\phi(0)=1$ and $\widehat\phi$ is smooth,
\[ \big|1-\widehat\phi(\varepsilon\xi)\big| \lesssim \min\{1,\varepsilon\abs{\xi}\}
\leq \varepsilon^\sigma |\xi|^\sigma \lesssim \varepsilon^\sigma (1+\xi^2)^{\sigma/2}, \]
so
\begin{align*}
\| g_m - g_m\ast_1\phi_\varepsilon \|_{\textup{H}^{-\sigma,0}}^2
& = \int_{\R^2} \big|\widehat{g_m}(\xi,\eta)\big|^2 \big| 1-\widehat\phi(\varepsilon\xi) \big|^2 (1+\xi^2)^{-\sigma} \dd \xi\dd \eta \\
& \lesssim \varepsilon^{2\sigma}\int_{\R^2}\abs{\widehat{g_m}(\xi,\eta)}^2\dd \xi\dd \eta
= \varepsilon^{2\sigma} \|g_m\|_{\textup{L}^2}^2.
\end{align*}
Hence, by \eqref{eq:local-L2-overlap} and $\|h\|_{\textup{H}^{0,-\sigma}}\leq\|h\|_{\textup{L}^2}$, the right hand side in \eqref{eq:rhsoftwob} is
\[ \lesssim \|f\|_{\textup{L}^\infty} \Bigl( \varepsilon^{2\sigma} \sum_{m\in\Z^2} \|g_m\|_{\textup{L}^2}^2 \Bigr)^{1/2}
\Bigl(\sum_{m\in\Z^2} \|h_m\|_{\textup{L}^2}^2\Bigr)^{1/2} 
\lesssim \varepsilon^\sigma \|f\|_{\textup{L}^\infty} \|g\|_{\textup{L}^2} \|h\|_{\textup{L}^2}. \]

The second term, $T_2$, is decomposed in the same way, leading to the study of
\begin{align*}
& \Big| \sum_{m\in\Z^2} \mathcal{N}\bigl(f_m, \,g_{m+k}\ast_1\phi_\varepsilon, \,h_{m+l} - h_{m+l}\ast_2\phi_\varepsilon\bigr) \Big| \\
& \lesssim \|f\|_{\textup{L}^\infty} \Bigl(\sum_{m\in\Z^2} \|g_m\ast_1\phi_\varepsilon\|_{\textup{H}^{-\sigma,0}}^2\Bigr)^{1/2}
\Bigl(\sum_{m\in\Z^2} \|h_m-h_m\ast_2\phi_\varepsilon\|_{\textup{H}^{0,-\sigma}}^2\Bigr)^{1/2} \\
& \lesssim \|f\|_{\textup{L}^\infty} \Bigl( \sum_{m\in\Z^2} \|g_m\|_{\textup{L}^2}^2\Bigr)^{1/2} \Bigl( \varepsilon^{2\sigma} \sum_{m\in\Z^2} \|h_m\|_{\textup{L}^2}^2\Bigr)^{1/2},
\end{align*}
where we used
\[ \| h_m - h_m\ast_2\phi_\varepsilon \|_{\textup{H}^{0,-\sigma}} \lesssim \varepsilon^{\sigma} \|h_m\|_{\textup{L}^2} \]
and \eqref{eq:local-L2-overlap}.
\end{proof}

The key point in the above proof is that the global Estimate \eqref{eq:more_general} was reduced to the local estimate from Lemma~\ref{lm:smoothing} via almost orthogonality \eqref{eq:local-L2-overlap}. We have not seen such a sharp localization argument in previous literature on point configurations, possibly because the target estimate is rarely independent of $\lambda$. In our case this owes to the fact that the anisotropic dilation $(x,y)\mapsto (\lambda x,y/\lambda)$ is area-preserving.

%%%%%

\subsection{Proof completion}

\begin{proof}[Proof of the upper estimate in Theorem~\ref{thm:main}]
Let
\begin{equation}\label{eq:deltaisdensity}
\delta := \frac{\abs{A}}{R^2} \in (0,1]
\end{equation}
and suppose that $A\subseteq[0,R]^2$ does not contain the pattern \eqref{eq:mainpattern} for any $x,y\in\R$, $t>0$.
Fix the exponent $\sigma>0$ from Lemma~\ref{lm:uniform}. 
Choose
\[ \varepsilon := c_0 \delta^{2/\sigma}, \]
where $c_0\in(0,1]$ will be chosen later.
By \eqref{eq:exact_zero} and Lemma~\ref{lm:uniform},
\begin{equation*}
\mathcal{N}^\varepsilon_\lambda(A) = \mathcal{N}^\varepsilon_\lambda(A) - \mathcal{N}^0_\lambda(A)
\lesssim R^2 \delta \varepsilon^{\sigma} \leq \frac{1}{2} c_1 R^2\delta^3 
\end{equation*}
for every $\lambda>0$, provided $c_0$ was chosen sufficiently small, where $c_1>0$ is the constant from Lemma~\ref{lm:structured}, i.e., 
\begin{equation*}
\mathcal{N}^1_\lambda(A)\geq c_1 R^2\delta^3
\end{equation*}
for every $\lambda\in[1/R,R]$.
In particular,
\[ \mathcal{N}^1_\lambda(A) - \mathcal{N}^\varepsilon_\lambda(A) \gtrsim R^2\delta^3. \]
We square this and integrate over $[1/R,R]$ with respect to $\dd\lambda/\lambda$, so that Lemma~\ref{lm:L2error} then gives
\[ R^4 \delta^6 \log R \lesssim \int_{1/R}^{R} \big| \mathcal{N}^\varepsilon_\lambda(A)-\mathcal{N}^1_\lambda(A) \big|^2\,\frac{d\lambda}{\lambda} \lesssim R^4 \delta^2 \log\frac{1}{\varepsilon}. \]
By the choice of $\varepsilon$ we have obtained
\[ \delta^4 \log R \lesssim 1+\log\frac{1}{\delta}, \]
i.e.,
\[ \frac{1}{\delta^4}\Bigl(1+\log\frac{1}{\delta}\Bigr) \gtrsim \log R. \]
This implies\footnote{If $x^4 (1+\log x) = y$ is solved for $x>0$ when $y>0$ is sufficiently large, then $x=((4+o(1))y/\log y)^{1/4}$, $y\to+\infty$.}
\[ \frac{1}{\delta} \gtrsim \Bigl( \frac{\log R}{\log\log R} \Bigr)^{1/4}, \]
which proves the more difficult half of Theorem~\ref{thm:main}.
\end{proof}

\begin{remark}
Note that Lemmas~\ref{lm:L2error} and \ref{lm:uniform} did not use any information about the localization of the set $A$. For instance, the proof of latter lemma owes this to the fact that transformations $(x,y)\mapsto(\lambda x,y/\lambda)$ are area-preserving, so one can in fact turn attention to a single value of $\lambda$ at a time, and all underlying estimates are exactly the same. An ultimate consequence is that we are able to obtain a better upper bound \eqref{eq:mainestimate} than the doubly-logarithmic bounds in \cite{B88nonlinear} and \cite{CDR21}. The only place where the concentration of $A$ on $[0,R]^2$ was, in fact, needed is the lower bound in Lemma~\ref{lm:structured}. 
\end{remark}

%%%%%

\section{Proof of the lower bound in Theorem~\ref{thm:main}}
\label{sec:lower_bound}

\begin{proof}[Proof of the lower estimate in Theorem~\ref{thm:main}]
For $R\geq 4$, let $m=\lfloor R/4\rfloor$ and define
\[ S_j := \Big\{ (x,y) \in [0,R]^2 \,:\, R-4j \leq x+y \leq R-4j+\frac{1}{8j} \Big\} \]
for $j=1,2,\ldots,m$.
Consider the set
\[ A_R := \bigcup_{j=1}^m S_j. \]
Thus, $A_R$ is a union of thin antidiagonal bands $S_j$, with the $j$th band having width $1/(8j)$ in the $x+y$ coordinate. It is illustrated in Figure~\ref{fig:exampleset}.

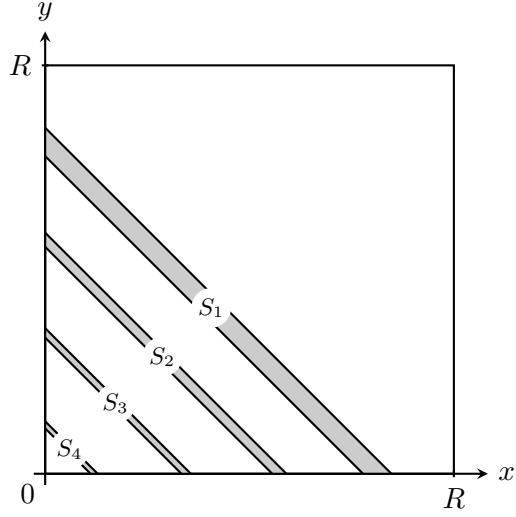
\begin{figure}
\centering
\begin{tikzpicture}[scale=0.3, >=stealth]
  \def\R{18}
  \pgfmathtruncatemacro{\m}{\R/4}
  \def\ExaggerationFactor{10} 
  \draw[->, thick, black] (-0.5, 0) -- (\R+1.5, 0) node[right, black] {$x$};
  \draw[->, thick, black] (0, -0.5) -- (0, \R+1.5) node[above, black] {$y$};
  \node[below left] at (0,0) {$0$};
  \begin{scope}
    \clip (0,0) rectangle (\R,\R);
    \foreach \j in {1,...,\m} {
      \pgfmathsetmacro{\cmin}{\R - 4*\j}
      \pgfmathsetmacro{\cmax}{\R - 4*\j + \ExaggerationFactor/(8*\j)}
      \fill[black!20] (\cmin, 0) -- (\cmax, 0) -- (0, \cmax) -- (0, \cmin) -- cycle;
      \draw[black, thick] (\cmin, 0) -- (0, \cmin);
      \draw[black, thick] (\cmax, 0) -- (0, \cmax);
      \pgfmathsetmacro{\midpoint}{(\cmin + \cmax)/2}
      \node[black, fill=white, inner sep=1pt, circle] 
        at (\midpoint/2, \midpoint/2) {\footnotesize $S_{\j}$};
    }
  \end{scope}
  \draw[thick] (0,0) rectangle (\R,\R);
  \draw[thick] (\R, 0.15) -- (\R, -0.15) node[below] {$R$};
  \draw[thick] (0.15, \R) -- (-0.15, \R) node[left] {$R$};
\end{tikzpicture}
  \caption{Example for the lower bound.}
  \label{fig:exampleset}
\end{figure}

On the one hand, since a band
\[ \{(x,y) \in [0,R]^2 : a \leq x+y \leq b\} \]
has area $(b^2-a^2)/2$ for any $0\leq a<b\leq R$, the whole set $A_R$ has measure
\begin{align*}
|A_R| & = \frac{1}{2}\sum_{j=1}^m \biggl( \Bigl(R-4j+\frac{1}{8j}\Bigr)^2 - (R-4j)^2 \biggr) \\
& = \frac{R}{8} \sum_{j=1}^m \frac{1}{j} - \frac{m}{2} + \frac{1}{128}\sum_{j=1}^m \frac{1}{j^2}
= \frac{1}{8} R \log R + O(R). 
\end{align*}
On the other hand, $A_R$ contains no triple of the form \eqref{eq:mainpattern}.
Namely, if it did contain such a triple for some $x,y\in\R$ and $t>0$, then we could find an index $1\leq j\leq m$ such that $(x,y)\in S_j$.
\begin{itemize}
\item If the points $(x+t,y)$ and $(x,y+1/t)$ also belonged to the same band $S_j$, we would arrive at a contradiction with $t\leq 1/(8j)$ and $1/t\leq 1/(8j)$.
\item If both points $(x+t,y)$ and $(x,y+1/t)$ rather belonged to $\cup_{k=1}^{j-1}S_k$, then the mutual separation of bands would lead to both $t>3$ and $1/t>3$, which is again a contradiction.
\item Finally, if $(x+t,y)\in S_j$ and $(x,y+1/t)\in \cup_{k=1}^{j-1}S_k$, then $t\leq 1/(8j)$ and $1/t<4j$, which is impossible again.
The same argument also applies when we interchange the roles of the two points and replace $t$ with $1/t$. 
\end{itemize}
This completes the proof of the lower estimate in \eqref{eq:mainestimate}.
\end{proof}

Let us remark that the argument above proves a lower bound for the upward-oriented triangles, i.e., when one restricts attention to positive $t$. If one also forbids triples with $t<0$, then the same construction does not apply, because it only controls ``forward differences'' in the $x+y$ coordinate.

%%%%%

\section{Proof of Theorem~\ref{thm:triangles}}

Let $\mathcal{R}_\theta$ denote the planar rotation by the angle $\theta$:
\[ \mathcal{R}_\theta(x,y) := (x\cos\theta-y\sin\theta, x\sin\theta+y\cos\theta). \]
For a bounded measurable set $A\subseteq\R^2$ we define the horizontal counting form
\[ \overrightarrow{\mathcal{M}}^0_{\lambda}(A) 
:= \int_0^{\infty}\!\int_{\R^3} \1_A(x,y) \1_A(x+\lambda u,y) \1_A\Bigl(x',y+\frac{1}{\lambda u}\Bigr)
\zeta(u) \dd x \dd x' \dd y \dd u \]
and its rotated directional variant
\[ \mathcal{M}^0_{\lambda}(A) := \frac{1}{2\pi} \int_0^{2\pi} \overrightarrow{\mathcal{M}}^0_{\lambda}(\mathcal{R}_{\theta}A) \dd\theta. \]
The corresponding smoothed versions of these forms are
\begin{align*} 
\overrightarrow{\mathcal{M}}^{\varepsilon}_{\lambda}(A) & := 
\int_0^{\infty}\!\int_{\R^3} \1_A(x,y) (\1_A\ast_1\phi_{\lambda\varepsilon})(x+\lambda u,y) (\1_A\ast_2\phi_{\lambda^{-1}\varepsilon})\Bigl(x',y+\frac{1}{\lambda u}\Bigr)
\zeta(u) \dd x \dd x' \dd y \dd u, \\
\mathcal{M}^{\varepsilon}_{\lambda}(A) & := \frac{1}{2\pi} \int_0^{2\pi} \overrightarrow{\mathcal{M}}^{\varepsilon}_{\lambda}(\mathcal{R}_{\theta}A) \dd\theta .
\end{align*}
Here, $\lambda\in(0,\infty)$ and $\varepsilon\in(0,1]$, just as before.

The following sections roughly mirror the technical parts of the proof for corners, but we still need to work out the details, since numerous quantitative savings are possible for general rotated triangles. Most notably, the Riesz energy of a set $A$ comes into play, besides its measure, as an additional refinement of the argument. 

%%%%%

\subsection{The structured part}

\begin{lemma}\label{lm:structured2}
For every measurable $A\subseteq[0,R]^2$ and every $\lambda$ with
\[ R\geq \max\{\lambda,\lambda^{-1}\}, \]
one has
\[ \mathcal{M}^1_\lambda(A) \gtrsim \frac{|A|^3}{R^3} = R^3 \Bigl(\frac{|A|}{R^2}\Bigr)^3. \]
\end{lemma}

\begin{proof}
It is sufficient to prove
\[ \overrightarrow{\mathcal{M}}^1_{\lambda}(A) \gtrsim \frac{|A|^3}{R^3}. \]
Afterwards, one simply applies this bound to the set $A$ rotated by angle $\theta$ which then still lies in an axes-aligned square with at most $\sqrt{2}$ times longer side.

By the form definition,
\[ \overrightarrow{\mathcal{M}}^1_\lambda(A) = \int_{\R^5} \1_A(x,y) \1_A(x'',y) \1_A(x',y') K^+_\lambda(x-x'',y-y') \dd x \dd x' \dd x'' \dd y \dd y', \]
where $K^+_\lambda$ is the same positive kernel from \eqref{eq:Kplus}.
Split $[0,R]$ into a collection $\mathcal{I}$ of $\sim R/\lambda$ many intervals $I$ of length $|I|\leq\lambda$. At the same time, partition $[0,R]$ into another collection $\mathcal{J}$ of $\sim \lambda R$ many intervals $J$ of length $|J|\leq 1/\lambda$.
Using \eqref{eq:Kplusbd} we bound
{\allowdisplaybreaks
\begin{align*} 
\overrightarrow{\mathcal{M}}^1_{\lambda}(A) 
& \gtrsim \sum_{\substack{I\in\mathcal{I}\\J\in\mathcal{J}}} \int_{I\times I\times [0,R]\times J\times J} \1_A(x,y) \1_A(x'',y) \1_A(x',y') \dd x \dd x'' \dd x' \dd y \dd y' \\
& = \sum_{J\in\mathcal{J}} |A\cap ([0,R]\times J)| \int_J \sum_{I\in\mathcal{I}} \Bigl( \int_I \1_A(x,y) \dd x \Bigr)^2 \dd y \\
& \gtrsim \frac{\lambda}{R} \sum_{J\in\mathcal{J}} |A\cap ([0,R]\times J)| \int_J \Bigl( \int_{[0,R]} \1_A(x,y) \dd x \Bigr)^2 \dd y \\
& \gtrsim \frac{\lambda^2}{R} \sum_{J\in\mathcal{J}} |A\cap ([0,R]\times J)|^3 \geq \frac{|A|^3}{R^3}.
\end{align*}
}
Here we used the Cauchy--Schwarz inequality followed by discrete Jensen's inequality.
\end{proof}

%%%%%

\subsection{The error part}

Let
\begin{equation}\label{eq:def_Riesz_en}
\mathcal{E}(A) := \int_{(\R^2)^2} \frac{\1_A(z) \1_A(z')}{|z-z'|} \dd z \dd z'
\end{equation}
denote the \emph{Riesz energy} of the set $A$.

\begin{lemma}\label{lm:L2error2}
For every bounded measurable $A\subseteq\R^2$ and every $\varepsilon\in(0,1]$,
\[ \Bigl( \int_0^{\infty} \abs{\mathcal{M}^\varepsilon_\lambda(A)-\mathcal{M}^1_\lambda(A)}^2 \,\frac{\dd\lambda}{\lambda} \Bigr)^{1/2}
\lesssim \mathcal{E}(A) \Bigl(\log\frac{1}{\varepsilon}\Bigr)^{1/2}. \]
\end{lemma}

\begin{proof}
The first part of the proof estimates the ``horizontal'' differences $\overrightarrow{\mathcal{M}}^\varepsilon_\lambda(A)-\overrightarrow{\mathcal{M}}^1_\lambda(A)$.
For shortness we introduce the measures of the horizontal sections of $A$: 
\[ G(y) := \int_{\R} \1_A(x,y) \dd x. \]
Proceeding as in the proof of Lemma~\ref{lm:L2error} we split
\[ \overrightarrow{\mathcal{M}}^\varepsilon_\lambda(A)-\overrightarrow{\mathcal{M}}^1_\lambda(A) = \overrightarrow{E}_{1}(\lambda,\varepsilon) + \overrightarrow{E}_{2}(\lambda,\varepsilon), \]
where
\begin{align*}
\overrightarrow{E}_{1}(\lambda,\varepsilon) 
&:= \int_0^{\infty}\!\int_{\R^2}
\1_{A}(x,y) \Bigl(\1_{A}\ast_1(\phi_{\lambda\varepsilon}-\phi_{\lambda})\Bigr)(x+\lambda u,y) \\
& \qquad\qquad\qquad \times (G\ast\phi_{\lambda^{-1}\varepsilon})\Bigl(y+\frac{1}{\lambda u}\Bigr) \zeta(u) \dd x \dd y \dd u, \\
\overrightarrow{E}_{2}(\lambda,\varepsilon)
&:= \int_0^{\infty}\!\int_{\R^2}
\1_{A}(x,y) (\1_{A}\ast_1\phi_{\lambda})(x+\lambda u,y) \\
& \qquad\qquad\qquad \times \Bigl(G\ast(\phi_{\lambda^{-1}\varepsilon} -\phi_{\lambda^{-1}})\Bigr)\Bigl(y+\frac{1}{\lambda u}\Bigr) \zeta(u) \dd x \dd y \dd u.
\end{align*}
We first address $\overrightarrow{E}_{2}$ estimating pointwise
\begin{align*} 
\big|\overrightarrow{E}_{2}(\lambda,\varepsilon)\big| 
& \leq \int_0^{\infty} \int_{\R} \Bigl( \underbrace{\int_{\R} \1_A(x,y) \dd x}_{G(y)} \Bigr) \Big| \Bigl(G\ast(\phi_{\lambda^{-1}\varepsilon} -\phi_{\lambda^{-1}})\Bigr)\Bigl(y+\frac{1}{\lambda u}\Bigr) \Big| \dd y \,\zeta(u) \dd u \\
& \lesssim \|G\|_{\textup{L}^2} \|G\ast(\phi_{\lambda^{-1}\varepsilon} -\phi_{\lambda^{-1}})\|_{\textup{L}^2} 
\end{align*}
and substituting $s=\lambda^{-1}$ to get
\begin{equation}\label{eq:Ev2}
\int_0^{\infty}\big|\overrightarrow{E}_2(\lambda,\varepsilon)\big|^2\,\frac{\dd\lambda}{\lambda}  
\lesssim \|G\|_{\textup{L}^2}^2 \int_0^{\infty} \big\|G\ast(\phi_{\varepsilon s} -\phi_{s})\big\|_{\textup{L}^2}^2 \frac{\dd s}{s} \lesssim \|G\|_{\textup{L}^2}^4 \log\frac{1}{\varepsilon}
\end{equation}
as a consequence of \eqref{eq:thephitrick}.
Next, we turn to $\overrightarrow{E}_{1}$. From
{\allowdisplaybreaks\begin{align*}
\big|\overrightarrow{E}_{1}(\lambda,\varepsilon)\big|
& \leq \int_0^{\infty} \int_{\R}
\Big| \int_{\R} \1_{A}(x,y) \Bigl(\1_{A}\ast_1(\phi_{\lambda\varepsilon}-\phi_{\lambda})\Bigr)(x+\lambda u,y) \dd x \Big| \\
& \qquad\qquad\qquad \times \Big| (G\ast\phi_{\lambda^{-1}\varepsilon})\Bigl(y+\frac{1}{\lambda u}\Bigr) \Big| \dd y \,\zeta(u) \dd u \\
& \leq \int_0^{\infty} \biggl( \int_{\R} \Big| \int_{\R} \1_{A}(x,y) \Bigl(\1_{A}\ast_1(\phi_{\lambda\varepsilon}-\phi_{\lambda})\Bigr)(x+\lambda u,y) \dd x \Big|^2 \dd y \biggr)^{1/2}\\
& \qquad\qquad \times \| G\ast\phi_{\lambda^{-1}\varepsilon} \|_{\textup{L}^2} \,\zeta(u) \dd u \\
& \lesssim \biggl( \int_{\R} \underbrace{\|\1_{A}(\cdot,y)\|_{\textup{L}^2}^2}_{G(y)} \big\| \Bigl(\1_{A}\ast_1(\phi_{\lambda\varepsilon}-\phi_{\lambda})\Bigr)(\cdot,y) \big\|_{\textup{L}^2}^2 \dd y \biggr)^{1/2} \|G\|_{\textup{L}^2}
\end{align*}}
we have
\begin{align}
\int_0^{\infty}\big|\overrightarrow{E}_1(\lambda,\varepsilon)\big|^2\,\frac{\dd\lambda}{\lambda}
& \lesssim \|G\|_{\textup{L}^2}^2 \int_{\R} G(y) \Bigl( \int_0^{\infty} \big\| \Bigl(\1_{A}\ast_1(\phi_{\lambda\varepsilon}-\phi_{\lambda})\Bigr)(\cdot,y) \big\|_{\textup{L}^2}^2 \frac{\dd\lambda}{\lambda}\Bigr) \dd y \nonumber \\
& \lesssim \|G\|_{\textup{L}^2}^2 \Bigl(\log\frac{1}{\varepsilon}\Bigr) \int_{\R} G(y) \,\| \1_{A}(\cdot,y) \|_{\textup{L}^2}^2 \dd y = \|G\|_{\textup{L}^2}^4 \log\frac{1}{\varepsilon}. \label{eq:Ev1}
\end{align}
Combining \eqref{eq:Ev2} and \eqref{eq:Ev1} we conclude
\begin{equation}\label{eq:notrotated}
\Bigl( \int_0^{\infty} \big|\overrightarrow{\mathcal{M}}^\varepsilon_\lambda(A)-\overrightarrow{\mathcal{M}}^1_\lambda(A)\big|^2 \,\frac{\dd\lambda}{\lambda} \Bigr)^{1/2} \lesssim \Bigl(\log\frac{1}{\varepsilon}\Bigr)^{1/2} \int_{\R} \Big| \int_{\R} \1_A(x,y) \dd x \Big|^2 \dd y.
\end{equation}

In the second part of the proof, we apply Estimate \eqref{eq:notrotated} to the set $A$ rotated by angle $\theta$ and average these over all $\theta\in[0,2\pi)$. That way we obtain
\begin{align*} 
& \Bigl( \int_0^{\infty} \big|\mathcal{M}^\varepsilon_\lambda(A)-\mathcal{M}^1_\lambda(A)\big|^2 \,\frac{\dd\lambda}{\lambda} \Bigr)^{1/2} \\
& \lesssim \Bigl(\log\frac{1}{\varepsilon}\Bigr)^{1/2} \frac{1}{2\pi} \int_0^{2\pi} \!\int_{\R} \Bigl( \int_{\R} \1_A(x\cos\theta+y\sin\theta, -x\sin\theta+y\cos\theta) \dd x \Bigr)^2 \dd y \dd \theta. 
\end{align*}
Now, we either use the so-called backprojection identity for the X-ray transform \cite{Helgason11}, or simply expand the right-hand side,
\begin{align*}
\Bigl(\log\frac{1}{\varepsilon}\Bigr)^{1/2} \frac{1}{2\pi} \int_0^{2\pi} \!\int_{\R} \int_{\R} \int_{\R} & \1_A(x\cos\theta+y\sin\theta, -x\sin\theta+y\cos\theta) \\
& \1_A(x'\cos\theta+y\sin\theta, -x'\sin\theta+y\cos\theta) \dd x \dd x' \dd y \dd \theta,
\end{align*}
and substitute 
\begin{align*}
& z_1 = x\cos\theta+y\sin\theta, \quad
z_2 = -x\sin\theta+y\cos\theta, \\
& z'_1 = x'\cos\theta+y\sin\theta, \quad
z'_2 = -x'\sin\theta+y\cos\theta, \\
& \implies \Big| \det \frac{\partial(z_1,z_2,z'_1,z'_2)}{\partial(\theta,x,x',y)} \Big| = |x - x'| = \Bigl((z_1-z'_1)^2+(z_2-z'_2)^2\Bigr)^{1/2},
\end{align*}
to turn it into
\[ \Bigl(\log\frac{1}{\varepsilon}\Bigr)^{1/2} \frac{1}{\pi} \int_{(\R^2)^2} \frac{\1_A(z_1,z_2)\1_A(z'_1,z'_2)}{((z_1-z'_1)^2+(z_2-z'_2)^2)^{1/2}} \dd z_1 \dd z_2 \dd z'_1 \dd z'_2. \qedhere \]
\end{proof}

It is interesting to observe that our proof of Lemma~\ref{lm:L2error2} used essentially that we are working with indicator functions. Otherwise we would not be able to equate the $\textup{L}^1$ norm $\|\1_A(\cdot,y)\|_{\textup{L}^1}$ with the square of the $\textup{L}^2$ norm $\|\1_A(\cdot,y)\|_{\textup{L}^2}^2$. 

%%%%%

\subsection{The uniform part}

\begin{lemma}\label{lm:uniform2}
There exists $\sigma>0$ such that, for every measurable set $A\subseteq[0,R]^2$, every $\lambda>0$, and every $0<\varepsilon\leq 1$,
\[ \big|\mathcal{M}^0_\lambda(A)-\mathcal{M}^\varepsilon_\lambda(A)\big| 
\lesssim \varepsilon^\sigma R |A| = \varepsilon^\sigma R^3 \frac{|A|}{R^2}. \]
\end{lemma}

\begin{proof}
Just as in the proof of Lemma~\ref{lm:structured2}, it is sufficient to prove the estimate
\[ \big|\overrightarrow{\mathcal{M}}^0_\lambda(A) - \overrightarrow{\mathcal{M}}^\varepsilon_\lambda(A)\big| 
\lesssim \varepsilon^\sigma R |A|, \]
apply it to the rotated set $A$, and finally integrate over the directions.

This time we introduce another function $g\colon\R^2\to\R$, 
\[ g(x,y) := \1_{[0,R]}(x) \1_{[0,R]}(y) \fint_{[0,R]} \1_A(x',y) \dd x', \]
which simply evaluates the averages of $\1_A$ over the horizontal sections of the containing square $[0,R]^2$.
It is supported on $[0,R]^2$ and takes values in $[0,1]$. Moreover,
\[ \overrightarrow{\mathcal{M}}^0_\lambda(A) = R \,\mathcal{N}^0_\lambda(\1_A,\1_A,g),\quad 
\overrightarrow{\mathcal{M}}^\varepsilon_\lambda(A) = R \,\mathcal{N}^\varepsilon_\lambda(\1_A,\1_A,g). \]
Thus, Estimate \eqref{eq:more_general} from the proof of Lemma~\ref{lm:uniform} applies and gives the desired bound.
\end{proof}

%%%%%

\subsection{Induction on scales}

\begin{proof}[Proof of Theorem~\ref{thm:triangles}]
Let $\delta$ denote the density of $A$ inside $[0,R]^2$, just as in \eqref{eq:deltaisdensity}.
After scaling the plane by $2^{-1/2}$, we can assume that $A$ contains no triple of points spanning a triangle of area $1/2$, which clearly implies 
\[ \mathcal{M}^0_\lambda(A)=0 \quad\text{for every }\lambda>0. \]
We apply Lemmas~\ref{lm:structured2}, \ref{lm:L2error2}, and \ref{lm:uniform2}, with $\varepsilon$ being a small multiple of $\delta^{2/\sigma}$. 
Integration of the square of
\[ \mathcal{M}^1_\lambda(A) - \mathcal{M}^\varepsilon_\lambda(A) \geq \frac{1}{2} \mathcal{M}^1_\lambda(A) \gtrsim R^3 \delta^3 \]
over $[1/R,R]$ with respect to $\dd\lambda/\lambda$, just as in the proof of Theorem~\ref{thm:main}, gives us 
\begin{equation}\label{eq:refinedenergy}
R^6 \delta^6 \log R \lesssim \mathcal{E}(A)^2 \Bigl(1+\log\frac{1}{\delta}\Bigr). 
\end{equation}
If we simply applied the Hardy--Littlewood--Sobolev inequality, 
\[ \mathcal{E}(A) \lesssim \|\1_A\|_{\textup{L}^{4/3}} \|\1_A\|_{\textup{L}^{4/3}} = |A|^{3/2} = R^3 \delta^{3/2}, \]
then \eqref{eq:refinedenergy} would yield
\[ \Bigl(\frac{1}{\delta}\Bigr)^{3} \Bigl(1+\log\frac{1}{\delta}\Bigr) \gtrsim \log R \]
and we would get
\begin{equation}\label{eq:mainestimate2_weak}
M_{\Delta}(R) \lesssim R^2 \Bigl(\frac{\log\log R}{\log R}\Bigr)^{1/3}.
\end{equation}
This is already an improvement over \eqref{eq:mainestimate}, but it is still weaker than \eqref{eq:mainestimate2}.

The idea is to rather prove by induction on scales $R\geq R_0$ that 
\begin{equation}\label{eq:induction_R}
|A| \leq C R^2 \Bigl(\frac{\log\log R}{\log R}\Bigr)^{1/2} 
\end{equation}
whenever $A\subseteq[0,R]^2$ contains no vertex set of a triangle of area $1/2$. Here $R_0\geq5$ and $C>0$ are sufficiently large absolute constants to be chosen later, while the induction means that we assume the estimate
\begin{equation}\label{eq:induction_r}
|A'| \leq C R'^2 \Bigl(\frac{\log\log R'}{\log R'}\Bigr)^{1/2}
\end{equation}
for every $R_0\leq R'\leq R/2$ and every measurable $A'$ with no triangles of area $1/2$ contained in an axes-aligned square of side $R'$. 
For every $z\in A$ write
\begin{align*} 
\int_{\R^2} \frac{\1_A(z') \dd z'}{|z-z'|} 
& = \int_{\R^2} \Bigl( \int_{|z-z'|}^{2R} \frac{\dd r}{r^2} + \frac{1}{2R} \Bigr) \1_A(z') \dd z'  \\
& = \int_{\R^2} \int_0^{2R} \frac{\1_A(z') \1_{\textup{D}(z,r)}(z')}{r^2} \dd r \dd z' + \frac{1}{2R} \int_{\R^2} \1_A(z') \dd z' \\
& = \int_0^{2R} \frac{|A\cap \textup{D}(z,r)|}{r^2} \dd r + \frac{|A|}{2R}.
\end{align*}
Applying the induction hypothesis \eqref{eq:induction_r} to $R'=2r$ and $A'=A\cap \textup{D}(z,r)$ for $R_0\leq r\leq R/4$ and estimating crudely for other values of $r\leq 2R$, we get
\begin{align*} 
\int_{\R^2} \frac{\1_A(z') \dd z'}{|z-z'|} 
& = \int_0^{R_0} \frac{|A\cap \textup{D}(z,r)|}{r^2} \dd r 
+ \int_{R_0}^{R/4} \frac{|A\cap \textup{D}(z,r)|}{r^2} \dd r
+ \int_{R/4}^{2R} \frac{|A\cap \textup{D}(z,r)|}{r^2} \dd r
+ \frac{|A|}{2R} \\
& \leq \pi R_0 + C \int_{5}^{R/4} \frac{(\log\log (2r))^{1/2}}{(\log (2r))^{1/2}} \dd r + 5R\delta,
\end{align*}
which is 
\[ \leq 5 R \biggl(\delta + C\Bigl(\frac{\log\log R}{\log R}\Bigr)^{1/2}\biggr) \]
as soon as $R_0$ and $C$ are sufficiently large. Integrating over $z\in A$ we get
\[ \mathcal E(A) \leq 5 R^3 \delta \biggl(\delta + C\Bigl(\frac{\log\log R}{\log R}\Bigr)^{1/2}\biggr). \]
Plugging this into \eqref{eq:refinedenergy} we obtain
\[ \delta^2 \leq C_0 \frac{(1+\log(1/\delta))^{1/2}}{(\log R)^{1/2}} \biggl(\delta + C\Bigl(\frac{\log\log R}{\log R}\Bigr)^{1/2}\biggr) \]
for an absolute constant $C_0>0$.
This gives
\[ \delta \leq C \Bigl(\frac{\log\log R}{\log R}\Bigr)^{1/2} \]
as soon as $C$ is sufficiently large in terms of $C_0$, which completes the proof of \eqref{eq:induction_R}.
\end{proof}

%%%%%%%%%%%%%%%%%%%%%%%%%%%%%%%%%%%%%%%%%%%%%%%%

\appendix

\section{Proof of the hyperbolic trilinear smoothing}
\label{sec:smoothing_proof}

The main part of the proof of Lemma~\ref{lm:smoothing} can be derived from the existing literature.
A rather general smoothing theorem for fiberwise bilinear multipliers was sketched by Hsu and Lin \cite{HL24}. However, we will prefer to use a less general formulation of this result from Lin's dissertation \cite{LinDiss}, which contains a very detailed proof in the form of a blueprint for a future Lean formalization.
Alternative proofs of Lemma~\ref{lm:smoothing} could be worked out by closely following the proofs of the trilinear smoothing inequalities by Christ, Durcik, and Roos \cite{CDR21} and Gaitan and Lie \cite{GaitanLie24}, on which the proof in \cite{LinDiss} is based.

Set 
\[ T(f_1,f_2)(x,y) := \int_{\R} f_1(x+t,y) f_2\Bigl(x,y+\frac{1}{t}\Bigr) \zeta(t) \dd t, \]
so that the desired estimate reads
\begin{equation}\label{eq:ts_desired_est}
\| T(f_1,f_2)\|_{\textup{L}^1} \lesssim \|f_1\|_{\textup{H}^{-\sigma,0}} \|f_2\|_{\textup{H}^{0,-\sigma}} 
\end{equation}
for an appropriately small $\sigma>0$ and $f_1,f_2$ supported in $[-10,10]^2$.

The fiberwise multiplier formulation of $T$ is
\begin{equation*}
T(f_1,f_2)(x,y) = \int_{\R^2} (\Fone f_1)(\xi,y) (\Ftwo f_2)(x,\eta) m(\xi,\eta)e^{2\pi i(x\xi+y\eta)} \dd\xi \dd\eta,
\end{equation*}
where
\begin{equation*}
m(\xi,\eta) := \int_{\R}e^{2\pi i(t\xi+\eta/t)} \zeta(t) \dd t .  
\end{equation*}

We localize the multiplier in frequency.

Let $\varphi\in C^{\infty}(\R)$ be a function such that $\1_{[-1,1]}\leq \varphi\leq \1_{[-2,2]}$. Let $\widetilde{\varphi}_0:=\varphi$ and for $j\ge 1$ let $\widetilde{\varphi}_j(u)=\varphi(2^{-j}u)-\varphi(2^{-j+1}u)$. 
Let $S(u):=\sum_{j\ge 0} \widetilde{\varphi}^2_j(u)$. At each point $u\in \R$, the sum defining $S(u)$ has at most two nonzero elements, so $S(u)\ge \frac12$. Therefore, if we define
\[ \psi_j(u):=\frac{\widetilde{\varphi}_j(u)}{S(u)^{1/2}}, \]
then $\psi_j\in C^{\infty}_c(\R)$ for every $j\ge 0$,
\[\operatorname{supp}\psi_0 \subseteq [-2,2], \quad  \operatorname{supp} \psi_j\subseteq 2^j([-2, -1/2]\cup [1/2,2]), \quad j\ge 1, \]
and
\[\sum_{j\ge 0}\psi_j^2\equiv1.\]
For $j,k\in \N_0$, let
\[m_{j,k}(\xi,\eta):=\psi_j^2(\xi)\psi_k^2(\eta) \int_{\R} e^{2\pi i(\xi t+\eta/t)}\zeta(t)\dd t\]
and define
\[ T_{j,k}(f_1,f_2)(x,y) :=\int_{\R^2}(\Fone f_1)(\xi,y)(\Ftwo f_2)(x,\eta)
m_{j,k}(\xi,\eta)e^{2\pi i(x\xi+y\eta)}\dd\xi\dd\eta . \]

The required Estimate \eqref{eq:ts_desired_est} will follow from the following two lemmas.

\begin{lemma}[Non-stationary part]
\label{lm:highlow}
Let $j,k\in \N_0$ be integers such that $|j-k|> 10$. Then the estimate
\begin{equation*}
\|T_{j,k}(f_1,f_2)\|_{\textup{L}^1} \lesssim_N 2^{-N\max\{j,k\}}\|f_1\|_{\textup{L}^2}\|f_2\|_{\textup{L}^2}
\end{equation*}
holds for every $N>0$.
\end{lemma}

\begin{proof}
The proof follows the proof of a single scale version of the high--low Coifman--Meyer paraproduct estimate.

Let $\Phi_{\xi,\eta}(t)=\xi t+\eta/t$. When $|j-k|> 10$, it follows that for $t\in \operatorname{supp}\zeta \subset [\frac{1}{2},2]$, $\xi \in \operatorname{supp}\psi_j$ and $\eta\in \operatorname{supp}\psi_k$ the following estimate holds
\begin{equation}
    \label{eq:phi_nonstat_est}
     |\Phi'_{\xi,\eta}(t)| = \Big|\xi - \frac{\eta}{t^2}\Big| \gtrsim 2^{\max\{j,k\}}.
\end{equation}
Repeated integration by parts with
\[ L_{\xi,\eta}:=\frac{1}{2\pi i\,\Phi'_{\xi,\eta}(t)}\frac{\dd}{\dd t} \]
gives, for every fixed finite range of $\alpha,\beta$ and every $N>0$,
\begin{equation}\label{eq:ts_nonstat_symb}
\big|\partial_\xi^\alpha\partial_\eta^\beta m(\xi,\eta)\big| \lesssim_{N,\alpha,\beta} 2^{-N\max\{j,k\}}.
\end{equation}
Indeed, differentiating the integral in $\xi$ and $\eta$ only inserts powers of $t$ and $t^{-1}$ with $\zeta(t)$ and the integration by parts gives the required decay.
Since $(\xi,\eta)\mapsto m(\xi,\eta)\psi_j(\xi)\psi_k(\eta)$ is a smooth function supported in a $[-2^{\max\{j,k\}+1}, 2^{\max\{j,k\}+1}]^2$ square (which has side length $2^{\max\{j,k\}+2}$), using classical results for Fourier series, it follows that
\[m(\xi,\eta)\psi_j(\xi)\psi_k(\eta)=\sum_{n_1,n_2\in \Z}c_{n_1,n_2}e^{2\pi i(n_1\xi + n_2\eta)/2^{\max\{j,k\}+2}},\]
where the coefficients, because of Estimate \eqref{eq:ts_nonstat_symb}, satisfy
\[|c_{n_1,n_2}|\lesssim_{\varphi} 2^{-N\max\{j,k\}}(1+|n_1|+|n_2|)^{-10}.\]
Therefore, using Cauchy--Schwarz and Young convolution inequalities, we have
{\allowdisplaybreaks\begin{align*}
&\|T_{j,k}(f_1,f_2)\|_{\textup{L}^1}\\
&\qquad\leq \sum_{n_1,n_2\in \Z} |c_{n_1,n_2}| \bigl\| f_1*_1\psi_j^{\vee}(x + n_1/2^{\max\{j,k\}+2},y) f_2*_2\psi_k^{\vee}(x,y + n_2/2^{\max\{j,k\}+2})\bigr\|_{\textup{L}^1}\\
&\qquad\lesssim_{\varphi} 2^{-N\max\{j,k\}}\sum_{n_1,n_2\in \Z}(1+|n_1|+|n_2|)^{-10}\|f_1\|_{\textup{L}^2}\|f_2\|_{\textup{L}^2}\\
&\qquad\lesssim 2^{-N\max\{j,k\}}\|f_1\|_{\textup{L}^2}\|f_2\|_{\textup{L}^2}. \qedhere
\end{align*}}
\end{proof}

\begin{lemma}[Stationary part]
\label{lm:highhigh}
There exists a constant $c>0$, depending only on $\zeta$ and $\psi$ in the definition of $m_{j,k}$, such that the following holds for every $j,k\in \N_0$ such that $|j-k|\leq 10$:
\begin{equation*}
\|T_{j,k}(f_1,f_2)\|_{\textup{L}^1([0,1]^2)}
\lesssim_{\zeta,\psi} 2^{-c\max\{j,k\}}\|f_1\|_{\textup{L}^2}\|f_2\|_{\textup{L}^2}.
\end{equation*}
\end{lemma}

\begin{proof}
We claim that this is essentially what Theorem \cite[Ch.~3, Thm.~3.1.12]{LinDiss} claims for $\gamma(t)=-1/t$, but a few comments are due.
First, we observe that
\[ \gamma'(t)=\frac{1}{t^2},\quad \gamma''(t)=-\frac{2}{t^3},\quad
\theta(t):=\frac{\gamma'(t)}{\gamma''(t)}=-\frac{t}{2}, \]
and hence
\[ \theta'(t)=-\frac12, \quad
\theta(t)\theta''(t)-\theta'(t)(1+\theta'(t))=\frac14, \]
so the quantitative nondegeneracy hypotheses in \cite[Thm.~3.1.12]{LinDiss} hold. 

Theorem \cite[Thm.~3.1.12]{LinDiss} is stated only for the case $j=k$ with the functions $\psi_j$ having support in smaller intervals (not annuli) $2^j[1-\delta,1+\delta]$ for $\delta>0$ sufficiently small depending on $\gamma$.
The assumption that $\delta>0$ is small enough is imposed because in the statement of theorem \cite[Thm.~3.1.12]{LinDiss}, the function $\gamma$ is defined only on $[1/2,3/2]$ and they have to make sure that the stationary point of the phase $\Phi$ lies in that interval so that the decomposition of the multiplier in Proposition \cite[Prop.~3.1.17]{LinDiss}, which uses the method of stationary phase, holds true.

In our case, the function $\gamma=-\frac{1}{t}$ is smooth on $(0,\infty)$ and satisfies the assumption 
\[\inf_{t\in I}|\gamma''(t)|\neq 0\] for any compact interval $I\subset (0,\infty)$. The choice $I=[2^{-6},2^{6}]$ will satisfy the requirements.

If $j,k\leq 20$, the statement trivially follows from Cauchy--Schwarz and Young convolution inequalities as in the proof of Lemma~\ref{lm:highlow}.\footnote{We prove this case separately to avoid the technicality that for $j=0$ or $k=0$, the truncated function $\psi_0 \1_{(0,\infty)}$, which we use in the rest of the proof is no longer smooth.}

For $\max\{j,k\}\ge 20$, the operators associated with 
\[ m_{j,k}(\xi,\eta) \1_{(0,\infty)}(\xi) \1_{(-\infty,0)}(\eta) \quad\text{and}\quad  
m_{j,k}(\xi,\eta) \1_{(-\infty,0)}(\xi) \1_{(0,\infty)}(\eta) \] 
can be handled using the same method as in Lemma~\ref{lm:highlow}. Indeed, the phase satisfies \eqref{eq:phi_nonstat_est}, allowing us to perform integration by parts.

Finally, to handle operators associated with 
\[ m_{j,k}(\xi,\eta) \1_{(0,\infty)}(\xi) \1_{(0,\infty)}(\eta) \quad\text{and}\quad 
m_{j,k}(\xi,\eta) \1_{(-\infty,0)}(\xi) \1_{(-\infty,0)}(\eta), \]
we observe that since $|j-k|\leq 10$, the stationary point of the phase $\Phi$ satisfies $t_0=\sqrt{\eta/\xi}\in [2^{-6},2^6]$. Therefore, the proof of \cite[Theorem~3.1.12]{LinDiss} with the given choice of $I$ gives the desired bound. 
\end{proof}

\begin{proof}[Proof of Lemma~\ref{lm:smoothing}]
As we have already observed, it remains to prove estimate \eqref{eq:ts_desired_est}.

Let $\widetilde{\psi}_0\in C_c^{\infty}(\R)$ be a function that is equal to $1$ on the support of $\psi_0$, let $\widetilde{\psi}_1\in C^{\infty}(\R)$ be a function supported in $[-4,-1/4]\cup [1/4,4]$ that is equal to $1$ on $ \operatorname{supp}\psi_1\subset [-2,-1/2]\cup[1/2,2]$ and let $\widetilde{\psi}_j(t)=\widetilde{\psi}_1(2^{-j}t)$. 
For $j\in \N_0$, let $P_j^{(1)}$ be the operator defined with 
\[ \bigl(P_j^{(1)}f\bigr)(x,y) 
= \Fone^{-1} \bigl( (\widetilde{\psi}_j\otimes \1_{\R}) (\Fone f)\bigr)(x,y)=\int_{\R}\widetilde{\psi}_j(\xi)(\Fone f)(\xi,y) e^{2\pi i \xi x} \dd \xi \]
and similarly let $P_j^{(2)}$ be the operator defined with
\[ \bigl(P_j^{(2)}f\bigr)(x,y)
= \Ftwo^{-1} \bigl((\1_{\R}\otimes\widetilde{\psi}_j) (\Ftwo f)\bigr)(x,y)=\int_{\R}\widetilde{\psi}_j(\eta)(\Ftwo f)(x,\eta) e^{2\pi i \eta y} \dd \eta. \]
Then, because of the support assumption, we conclude that
\[T_{{j,k}}(f_1,f_2)=T_{{j,k}}(P_j^{(1)}f_1,P_k^{(2)}f_2).\]

Observe that for functions $f_1,f_2$ supported on $[-10,10]^2$, we have $\operatorname{supp}T(f_1,f_2)\subset [-20,20]^2$. Therefore, by partitioning $[-20,20]^2$ into $O(1)$ squares of sidelength $1$, it is sufficient to bound
$\|T(f_1,f_2)\|_{L^1(Q)}$ for a square $Q$ of sidelength $1$. Without loss of generality, we can assume that $Q=[0,1]^2$.

Using the multiplier decomposition, previous observations, Lemmas~\ref{lm:highlow} and \ref{lm:highhigh}, and the fact that $2\max\{j,k\}\ge  j+k$, there exists $c>0$ such that the following chain of inequalities is true
{\allowdisplaybreaks\begin{align*}
\|T(f_1,f_2)\|_{\textup{L}^1([0,1]^2)}
&\leq \sum_{j,k\ge 0} \|T_{{j,k}}(f_1,f_2)\|_{\textup{L}^1([0,1]^2)}\\ 
&= \sum_{j,k\ge 0} \|T_{{j,k}}(P_j^{(1)}f_1,P_k^{(2)}f_2)\|_{\textup{L}^1([0,1]^2)}\\
&\lesssim_{\zeta,\varphi} \sum_{j,k\ge 0}2^{-c(j+k)}\|P_j^{(1)}f_1\|_{\textup{L}^2}\|P_k^{(2)}f_2\|_{\textup{L}^2}\\
&=\Bigl(\sum_{j\ge 0}2^{-cj}\|P_j^{(1)}f_1\|_{\textup{L}^2}\Bigr) \Bigl(\sum_{k\ge 0}2^{-ck}\|P_k^{(2)}f_2\|_{\textup{L}^2}\Bigr).
\end{align*}}
Finally, using the Cauchy--Schwarz inequality and the Littlewood--Paley characterization of the Sobolev norm, we have
\[\sum_{j\ge 0}2^{-cj}\|P_j^{(1)}f_1\|_{\textup{L}^2}\leq \underbrace{\Bigl(\sum_{j\ge 0} 2^{-cj} \Bigr)^{1/2}}_{\lesssim 1} \Bigl(\sum_{j\ge 0}2^{-cj}\|P_j^{(1)}f_1\|_{\textup{L}^2}^2 \Bigr)^{1/2}\lesssim_c \|f_1\|_{\textup{H}^{-c/2,0}}.\]
Using the analogous inequality for $f_2$, we conclude the proof.
\end{proof}

%%%%%

\section{Graham's argument}
\label{sec:Grahams}

The following lemma is the most likely reconstruction of Graham's comment in \cite{Gra80} based on Szemer\'{e}di's theorem.

\begin{lemma}\label{lm:Grahams_lemma}
For every $\beta\in(0,1]$, there exist a positive integer $n$ and a number $T>0$ such that every set $B\subseteq \gridn^2$ with at least $\beta n^2$ elements contains three points spanning a triangle of area exactly $T$.
\end{lemma}

\begin{proof}
Take a positive integer $r\geq 4/\beta$. By Szemer\'{e}di's theorem there exists a positive integer $N$ such that every subset of $\{0,1,\dots,N-1\}$ of size at least $(\beta/2) N$ contains a non-degenerate arithmetic progression of length $r!+1$.
We choose the integer $T$ mentioned in the lemma formulation to be $T:=r!N!/2$. Finally, let $n$ be a multiple of $N$ so large that $n\geq(r+1)N!$.

Let $B\subseteq\gridn^2$ satisfy $|B|\geq \beta n^2$. A row 
\[ \big\{(x,y) : x\in\gridn\big\} \] 
indexed by $y\in \gridn$ is said to be \emph{populated} if at least $(\beta/2)n$ of its $n$ elements belong to $B$. There are at least $(\beta/2)n$ populated rows, since otherwise
\[ |B| \leq \frac{\beta}{2}n^2 + \Bigl(n-\frac{\beta}{2}n\Bigr)\frac{\beta}{2}n < \beta n^2, \]
which would be in a contradiction with our assumption on the size of $B$.
Partition all rows from the $n\times n$ grid into $n/N$ strips of height $N$. At least one strip contains at least $(\beta/2)N$ populated rows. By the choice of $N$ according to Szemer\'{e}di's theorem, there are $r!+1$ equally spaced populated rows inside that strip; let them be determined with their vertical coordinates
\begin{equation}\label{eq:list_of_ys}
y_0,\ y_0+l,\ y_0+2l,\ \ldots,\ y_0+r!l. 
\end{equation}
We define $k:=N!/l\in \N$.

Now consider the lowest populated row, the one indexed by $y_0$: the set $B$ occupies at least $(\beta/2)n$ of its points.
Consider all arithmetic progressions of length $r+1$, 
\[ x_0,\ x_0+k,\ x_0+2k,\dots,\ x_0+rk, \]
that fully belong to $\gridn$.
We claim that at least one of these progressions contains at least two numbers from $\{x : (x,y_0)\in B\}$.
If that was not the case, then we would fully cover (with possible overlaps) $\gridn$ by $\lceil n/(r+1)\rceil$ such arithmetic progressions to conclude that the row indexed by $y_0$ contains at most
\[ \Big\lceil \frac{n}{r+1}\Big\rceil < \frac{2n}{r} \leq \frac{\beta}{2}n \]
points from $B$, which contradicts the choice of $y_0$.
Choose such a progression and such two numbers $x_0+ik$, $x_0+jk$ for some integers $x_0$ and $0\le i<j\le r$. Thus, 
\begin{equation}\label{eq:pts2}
(x_0+ik,y_0),\ (x_0+jk,y_0) \in B.
\end{equation}
The number
\[ y_1=y_0+\frac{r!}{j-i}l \]
is one of the numbers \eqref{eq:list_of_ys}, so we can find a point
\begin{equation}\label{eq:pts1}
(x_1,y_1)\in B.
\end{equation}
The three points \eqref{eq:pts2}, \eqref{eq:pts1} determine a triangle of area
\[ \frac{1}{2} (j-i)k \frac{r!}{j-i}l = \frac{1}{2} r! N! = T. \qedhere \]
\end{proof}

Note that we could not have been very picky about the area value $T$ in the previous lemma. Namely, it had to satisfy certain divisibility properties: if $2T$ had too few divisors, then it would not allow us to consider sufficiently many potential triangles with vertices on the integer lattice $\Z^2$. However, the actual values of $n$ and $T$ will not matter when we pass from the discrete setting to $\R^2$.

We briefly argue how Lemma~\ref{lm:Grahams_lemma} gives an alternative proof of \eqref{eq:Graham_weak}, as this transference is easy and standard.
Namely, fix $\delta\in(0,1]$. Apply Lemma~\ref{lm:Grahams_lemma} with $\beta=\delta/2$, which then gives $n\in\N$ and $T>0$ with the stated properties. We will show that, for every $R\geq 8n T^{-1/2} \delta^{-1}$, every subset $A\subseteq[0,R]^2$ of measure greater than $\delta R^2$ contains a triple of points that span a triangle of area $1$.
Namely, for every 
\[ (u,v)\in Q:= \Bigl[ 0,R-\frac{n-1}{T^{1/2}} \Bigr]^2 \] 
define
\[ B_{u,v} := \big\{ (k,l)\in\gridn^2 : (u,v) + T^{-1/2}(k,l) \in A \big\}. \]
Then
\[ \fint_{Q} |B_{u,v}| \dd u \dd v 
\geq \frac{1}{R^2} n^2 \Big| A \cap \Bigl[\frac{n}{T^{1/2}},R-\frac{n}{T^{1/2}}\Bigr]^2 \Big|
> \frac{\delta}{2}n^2, \]
so there exist $u,v$ such that $|B_{u,v}|>(\delta/2)n^2$. From Lemma~\ref{lm:Grahams_lemma} we know that $B_{u,v}$ contains vertices of a triangle with area $T$, whence $A$ contains vertices of a triangle with area $1$. Consequently, $M_{\Delta}(R)\leq \delta R^2$.

\begin{remark}
The latter proof gives quantitatively very weak upper bounds on $M_{\Delta}(R)$ in comparison with our earlier result.
Namely, by the best currently known bounds in Szemer\'{e}di's theorem \cite{LSS24}, already the number $N$ used in the proof grows like 
\[ N\geq\exp \Bigl( \exp \Bigl( (\log(2/\beta))^{C(\lceil4/\beta\rceil!+1)} \Bigr) \Bigr), \]
where the dependence of the constants $C(k)$ on the progression length $k$ is not even easy to track down explicitly.
\end{remark}

%%%%%%%%%%%%%%%%%%%%%%%%%%%%%%%%%%%%%%%%%%%%%%%%

\section*{Declaration of AI usage}
OpenAI's \emph{ChatGPT} 5.4 Pro was used to construct the example that gives the lower bound $\Omega(R\log R)$. \emph{ChatGPT} 5.5 Pro drafted an improvement over \eqref{eq:mainestimate2_weak}, which led us to \eqref{eq:mainestimate2}. It was also used to clarify Graham's cryptic remark and reconstruct its intended proof. Google's \emph{Gemini} 3.1 Pro was used to draw Figure~\ref{fig:exampleset}. The ideas, the proofs, and the actual writing of the manuscript are entirely the work of the authors, who assume full responsibility for the content.

%%%%%%%%%%%%%%%%%%%%%%%%%%%%%%%%%%%%%%%%%%%%%%%%

\section*{Acknowledgments and funding}
We are grateful to Michael Christ for several useful discussions, to Adrian Beker and Polona Durcik for pointing us to the relevant literature, and to Fred Lin for providing a copy of the doctoral dissertation \cite{LinDiss}.

This work was supported in part by the Croatian Science Foundation under the project HRZZ-IP-2022-10-5116 (\emph{FANAP}). The paper was also supported in part by the European Union -- NextGenerationEU through the National Recovery and Resilience Plan 2021--2026, via an institutional grant of the University of Zagreb Faculty of Science, IK IA 1.1.3, \emph{Impact4Math}.

%%%%%%%%%%%%%%%%%%%%%%%%%%%%%%%%%%%%%%%%%%%%%%%%

\bibliographystyle{plainurl}
\bibliography{right_area1_triangles}

\end{document}